\documentclass[a4paper, 11pt, BCOR=2mm]{scrartcl}
\usepackage[utf8]{inputenc}
\usepackage[T1]{fontenc}
\usepackage{lmodern}
\usepackage{amsmath}
\usepackage{amsthm} % must be loaded *before* cleveref, but *after* amsmath
\usepackage{hyperref}
\hypersetup{colorlinks = true, citecolor = blue}
\usepackage{cleveref}
\usepackage{appendix}
\usepackage[english]{babel}
\usepackage{amsfonts, amssymb, mathtools, mathrsfs, stmaryrd}
\usepackage{tikz}
\usetikzlibrary{backgrounds}
\usepackage{subfiles}

\theoremstyle{plain}
  \newtheorem{thm}{Theorem}[section]%[subsection]
\theoremstyle{plain}
  \newtheorem{lem}[thm]{Lemma}
  \newtheorem{cor}[thm]{Corollary}
  \newtheorem{prop}[thm]{Proposition}
  \newtheorem{fact}[thm]{Fact}
\theoremstyle{definition}
  \newtheorem{defi}{Definition}[section]%[subsection] 
\theoremstyle{remark}
  
  \newtheorem{exe}{Example}[section]%[subsection]

\crefname{lem}{Lemma}{Lemmas}
\Crefname{lem}{Lemma}{Lemmas}
\crefname{thm}{Theorem}{Theorems}
\Crefname{thm}{Theorem}{Theorems}
\crefname{cor}{Corollary}{Corollaries}
\Crefname{cor}{Corollary}{Corollaries}
\crefname{prop}{Proposition}{Propositions}
\Crefname{Prop}{Proposition}{Propositions}
\crefname{section}{Section}{Sections}
\Crefname{Section}{Section}{Sections}
\crefname{subsection}{Subsection}{Subsections}
\Crefname{Subsection}{Subsection}{Subsections}

\numberwithin{equation}{section}
\newenvironment{keywords}{\par\noindent\emph{Keywords: }}{}
\newenvironment{amscat}{\par\noindent\emph{AMS 2010 subject
  classifications:}\par}{}
\newenvironment{versionnotes}{\par\noindent\begingroup\small
  \emph{Version Notes: }}{\endgroup}
\newcommand\interval[4]{\mathopen{#1}#2\mathclose{},#3\mathclose{#4}}
\newcommand\intff[2]{\interval{[}{#1}{#2}{]}}

\newcommand\intoo[2]{\interval{(}{#1}{#2}{)}}

\newcommand\N{\mathbb{N}}
\newcommand\Ne{\N^{*}}
\newcommand\R{\mathbb{R}}

\newcommand\words{\mathcal{U}}
\newcommand\infinitewords{\words_{\infty}}
\renewcommand\root{\text{\o}}
\newcommand\prt[2][]{{#2}_{*#1}}
\newcommand\rays[1]{\partial{#1}}

\newcommand\noleaf{\mathscr{T}}
\newcommand\weightedtrees{\mathscr{T}_w}
\newcommand\weightedchopped{\weightedtrees^{\leq}}
\newcommand\weightedtreeswithrays{\mathscr{T}_{w,r}}
\newcommand\weightedtreeswithpath{\mathscr{T}_{w,p}}
\newcommand\gluetrees{\mathbin{\triangleleft}}
\newcommand\wt[1][t]{A_{#1}}
\newcommand\numch[1][t]{\nu_{#1}}

\newcommand\gcp{\wedge}
\newcommand\Se{\mathrm{S}_{\mathsf{e}}}
\newcommand\Sr{\mathrm{S}_{\mathsf{r}}}

\newcommand\St{\tilde{S}}
\newcommand\Shift{\mathrm{S}}
\newcommand\harm{\mathsf{HARM}}
\newcommand\unif{\mathsf{UNIF}}
\newcommand\GW{\mathbf{GW}}
\newcommand\pp{\mathbf{p}}
\newcommand\mureg{\mu_{\mathsf{r}}}
\newcommand\muex{\mu_{\mathsf{e}}}
\DeclarePairedDelimiter{\abs}{\lvert}{\rvert}
\newcommand\setof[2]{\left\{#1: #2 \right\}}
\newcommand\dd{\mathop{}\!\mathrm{d}}
\newcommand\indic[1]{\mathbf{1}_{\left\lbrace#1\right\rbrace}}
\newcommand\indi[1]{\mathbf{1}_{#1}}
\newcommand\compl[1]{{#1}^{\mathrm{c}}}
\let\P\relax
\DeclareMathOperator\P{\mathbb{P}}
\DeclareMathOperator\E{\mathbb{E}}
\DeclareMathOperator{\Pq}{P} % quenched
\DeclareMathOperator{\Eq}{E} % quenched
\DeclareMathOperator{\Et}{\mathbf{E}} % tree
\DeclareMathOperator{\Pt}{\mathbf{P}} % tree
\DeclareMathOperator{\Ea}{\E} %annealed
\DeclareMathOperator{\Pa}{\P} %annealed

\newcommand\brcd[2]{\left[#1 \mathrel{}\middle| \mathrel{} #2 \right]}
\newcommand\prcd[2]{\left(#1 \mathrel{}\middle| \mathrel{} #2 \right)}
\newcommand\et{\mathsf{et}} % exit time
\newcommand\ep{\mathsf{ep}} % exit point
\newcommand\ft{\mathsf{ft}} % fresh time
\newcommand\fp{\mathsf{fp}} % fresh point
\newcommand\rt{\mathsf{rt}} % regeneration time
\newcommand\rp{\mathsf{rp}} % regeneration point
\newcommand\rh{\mathsf{rh}} % regeneration height
\DeclareMathOperator\ray{ray} % ray defined by transient path
\newcommand\Xa{\mathbf{X}} % path
\newcommand\xa{\mathbf{x}}
\newcommand\Ya{\mathbf{Y}} % path

\DeclareMathOperator\hausd{\dim_{\mathrm{H}}}
\DeclareMathOperator\distrays{d_{\infinitewords}}

\newcommand\visw{\overline{\mathsf{VIS}}}
\newcommand\vis{\mathsf{VIS}}
\renewcommand{\tilde}{\widetilde}
\title{Dimension Drop for Transient Random Walks on Galton-Watson Trees in
Random Environments}
  \author{Pierre Rousselin\thanks{LAGA, University Paris 13; Labex MME-DII}}
\date{\today}
\begin{document}
\maketitle
\begin{abstract}
  We prove that the dimension drop phenomenon holds for the harmonic measure associated to
  a transient random walk in a random environment (as defined in
  \cite{faraud_2011, lyons_pemantle_rwre})
  on an infinite Galton-Watson tree without leaves.
  We use regeneration times and ergodic theory 
  techniques from \cite{LPP_biased} to give an explicit construction of the
  invariant measure for the forward environment seen by the particule at exit
  times which is absolutely continuous with respect to the joint law of the tree
  and the path of the random walk.
\end{abstract}
% {{{ keywords, ams category, version notes
\begin{amscat}
  Primary 60J50, 60J80, 37A50;
  secondary 60J05, 60J10.
\end{amscat}
\begin{versionnotes}
  This is version 1.0. Any feedback is highly appreciated. 
\end{versionnotes}
% }}}

\section*{Introduction}
\label{sec:intro}
Consider an infinite rooted tree $t$.
A ray in $t$ is an infinite path starting from its root, that never
backtracks.
The set of all rays in $t$ is called its boundary and denoted
$\rays t$. We put a natural topology on $\rays t$ by saying that two rays 
are close to each other when they coincide in a large ball centered at the root
of $t$.

We are interested in natural Borel probability measures on $\rays t$.
In many cases, such a measure $\theta$ happens to have full support.
However, it may turn out that $\theta$-almost every ray $\xi$ belongs to a
Borel subset $C_\theta$ of $\rays t$ that is ``small'', compared to the whole
boundary. To make this rigorous and quantitative, we put a suitable metric on
the boundary $\rays t$ and use Hausdorff dimensions. The Hausdorff
dimension of the measure $\theta$ is the Hausdorff dimension of a minimal 
Borel subset $C_\theta$ such that $\theta \left(C_\theta \right) = 1$. 
One says
that the \emph{dimension drop phenomenon} occurs for $\theta$ when this
dimension is strictly less than the dimension of the whole boundary.

Now, let us suppose that we have a transition matrix $\Pq^t$ on the vertices of
$t$ that defines a nearest-neighbour random walk on $t$, which is
\emph{transient}. By transience, almost every trajectory of the walk ``goes to
infinity'' and shares infinitely many vertices with a unique random ray $\Xi$.
The law of this random ray $\Xi$ is called the \emph{harmonic measure} on $\rays
t$, with respect to the transition matrix $\Pq^t$.

When the tree $T$ is a random, Galton-Watson tree, the dimension drop phenomenon
for the harmonic measure is already known to occur in some cases : see \cite{LPP95} for
the simple random walk, \cite{LPP_biased} for transient $\lambda$-biased random
walks and \cite{Curien_LeGall_harmonic, Shen_harmonic_infinite_variance,
rousselin_2017A} for other models depending on random
lengths on the Galton-Watson tree.

In this work, we prove that the dimension drop phenomenon occurs for the harmonic
measure with respect to a transition matrix defined by a 
\emph{random environment}
on the Galton-Watson tree.
This model was introduced in
\cite{lyons_pemantle_rwre} and has been extensively studied (see for instance
\cite{aidekon_2007, hu2016slow, andreoletti_debs, andreoletti_chen}).
We will use the definition from \cite{faraud_2011}, which is a generalization
and can be described as follows.
Let $\pp \coloneqq (p_k)_{k\geq 1}$ a sequence of non-negative real numbers such that
$p_1 < 1$ and $ \sum_{k\geq 1}^{} p_k = 1$. Assume that
$m \coloneqq \sum_{k\geq 1}^{} k p_k < \infty$.
Let $(N, A)$ a random element
in $\Ne\times\bigcup_{k\geq 1}\intoo{0}{\infty}^k$ such that the marginal
law of $N$ is $\pp$ and, for
any $k\geq 1$, if $N = k$, then $A$ lies in $\intoo{0}{\infty}^k$.
Build a random weighted Galton-Watson tree $T$ in the following way :
we start from a root $\root$ and pick a random element $(N,A)$. 
The number of children of the root is then $N$ and the children
$1$, $2$, \dots, $N$ of the root have respective weights
$A(1)$, $A(2)$, \dots, $A(N)$, where $\left(A(1), \dotsc, A(N)\right)
\coloneqq A$. Then, continue recursively in an independant manner on the
subtrees starting from the vertices $1$, \dots, $N$.

Now, conditionally on $T$, we start a nearest-neighbour
random walk $\Xa$ from $\root$ such that, if the walk is at a vertex
$x$ of $T$, the walk may jump to one of the children of $x$ with probability
equal to the weight of this child
divided by the sum of the weights of the children of $x$
plus $1$ while it goes back to the parent of $x$ with probability equal 
to 1 divided by
the same sum.
We know, from \cite{lyons_pemantle_rwre} and \cite{faraud_2011} a transience
criterion for this model, and assume throughout this work that we are in this regime.
Our main result is the following theorem:
\begin{thm}[Dimension drop for $\harm$]
  \label{thm:ddrop}
  Let $T$ a random weighted Galton-Watson tree.
  The harmonic measure $\harm_T$ is almost surely exact-dimensional and
  its Hausdorff dimension
  is almost surely a constant that equals
  \begin{equation*}
    \hausd \harm_T
    = \frac{1}{\Et \left[\kappa (T)\right]} 
    \E\left[ -\log \left(\harm_T \left(\Xi_1\right)\right) \kappa (T) \right],
  \end{equation*}
  with $\kappa$ defined by \eqref{eq:defbetakappa}.
  It is almost surely strictly less than the Hausdorff dimension of the whole
  boundary $\partial T$ (which is almost surely $\log m$), 
  unless the model reduces to a transient $\lambda$-biased random walk 
  (with a deterministic and constant $\lambda < m$) on an $m$-regular tree.
\end{thm}

Our results are inspired by the work of Lyons, Peres and Pemantle on transient
$\lambda$-biased random walks on Galton-Watson trees (\cite{LPP_biased}).
We use in the same way the notions of exit times and regeneration times to build
an invariant measure for the forward environment seen by the particle at exit times. The
construction of this measure, via a Rokhlin tower, was already suggested in
\cite{LPP_biased}.

The paper is organized as follows.
In \cref{sec:preliminaries}, we introduce our notations and define
the space of weighted trees, the transient paths on such trees,
the exit times and
regeneration times of such paths.
We also recall a transience
criterion by Lyons and Pemantle (\cite{lyons_pemantle_rwre}), generalized by Faraud (\cite{faraud_2011}).
In \cref{sec:regeneration}, we show that there are almost surely infinitely many
regeneration times and find an invariant measure for the forward environment
seen by the particle at such times. Again, this follows the ideas of
\cite{LPP_biased}, but we give detailed proofs in our setting of weighted trees
for completeness.
The heart of this work is \cref{sec:invariant_exit}, where we give a detailed 
``tower construction'' over the preceding dynamical system to build an invariant measure
for the forward environment seen by the particle at exit times. 
We then show that this measure has a density with respect to the joint law of the tree and the random
path on it and give an expression of this density.
We conclude in \cref{sec:invariant_harm} by projecting this measure on the
space of trees with a random ray on it and using the general theory of flow
rules on Galton-Watson trees developed in \cite{LPP95}.

\paragraph{Acknowledgements.}
The author is very grateful to his Ph.D. supervisors 
Julien Barral and Yueyun Hu for many
interesting discussions and constant help and support.
He also thanks \'Elie A\"idekon for encouraging him to write
this work.

\section{Notations and preliminaries}
\label{sec:preliminaries}

\subsection{Words and paths}

Let $\Ne \coloneqq \left\lbrace 1, 2, \dotsc \right\rbrace $ be our alphabet.
The set of all finite words on $\Ne$ is
\begin{equation*}
  \words \coloneqq \bigsqcup_{k = 0}^{\infty} \left(\Ne\right)^k,
\end{equation*}
with the convention
$\left(\Ne\right)^0 \coloneqq \left\lbrace\root\right\rbrace$,
$\root$ being the empty word.
The length $\abs{x}$ of a word $x$ is the unique integer such that $x$ belongs to
$\left(\Ne\right)^{\abs{x}}$.
The concatenation of the words $x$ and $y$ is denoted by $x y$.
A word 
$x = (i_1, i_2, \dots, i_{\abs{x}})$ 
is called a prefix of a word 
$y = (j_1, j_2, \dots, j_{\abs{y}})$
when either $x = \root$ or
$\abs{x} \leq \abs{y}$ and $i_\ell = j_\ell$ for any $\ell \leq \abs{x}$.
We denote
$\preceq$ this partial order and $x \gcp y$ the greatest common prefix of
$x$ and $y$. The parent of a non-empty word 
$x = (i_1, i_2, \dots, i_{\abs{x}})$ 
is
$\prt{x} \coloneqq (i_1 , i_2 , \dotsc , i_{\abs{x}-1})$ if
$\abs{x} \geq 2$; 
otherwise
it is the empty word $\root$. 
We also say that $x$ is a child of $\prt{x}$.
When a finite word $x$ is a prefix of a word $y$, we  define
$x^{-1} y$ as the unique word $z$ such that $y = xz$.

We add an artifical parent of the empty word.
Let $\prt\root$ an arbitrary element that does not belong to
$\words$ and 
$\prt\words \coloneqq \words \cup \left\lbrace \prt\root \right\rbrace$.
We let
$\abs{\prt\root} \coloneqq -1$, $ \prt{\left(\root\right)} \coloneqq \prt\root$
and, for all $x$ in $\words$, $x \prt\root = \prt\root x = \prt{x}$.

An infinite path in $\prt\words$ is a sequence
$
  \xa = \left(x_0, x_1, \dotsc\right)
$
such that for any $k \geq 0$, $x_{k+1}$ is either a child of $x_k$ or its
parent.
A \emph{transient path} is an infinite path $\xa$ such that
$\lim_{k \to \infty} \abs{x_k} = \infty$.

For such a path $\xa$, we define :
\begin{itemize}
  \item the set of \emph{fresh times} :
    \begin{equation*}
      \ft \left(\xa\right)
      \coloneqq
      \setof{s \geq 0}{ \forall k < s, x_k \neq x_s} 
      \eqqcolon \left\lbrace \ft_0 \left(\xa\right), 
      \ft_1 \left(\xa\right), \dots \right\rbrace,
    \end{equation*}
    where $\ft_0 \left(\xa\right) < \ft_1 \left(\xa\right) < \dotsb$ ;
  \item the set of \emph{exit times}
    \begin{equation*}
      \et \left(\xa\right)
      \coloneqq
      \setof{s \geq 0}{ \forall k > s, x_k \neq \prt{\left(x_s\right)}} 
      \eqqcolon 
      \left\lbrace
        \et_0 \left(\xa\right), \et_1 \left(\xa\right), \dots 
      \right\rbrace,
    \end{equation*}
    where $\et_0 \left(\xa\right) < \et_1 \left(\xa\right) < \dotsb$ ;
  \item the \emph{exit points},
    $\ep_k \left(\xa\right) \coloneqq x_{ \et_k \left(\xa\right) }$, for
    $k = 0, 1, \dotsc$ ;
  \item the set of \emph{regeneration times} :
    \begin{equation*}
      \rt \left(\xa\right) \coloneqq \ft \left(\xa\right) \cap
      \et \left(\xa\right).
    \end{equation*}
    Likewise, the regeneration times (if there are any) are ordered
    $
      \rt_0 \left(\xa\right) < \rt_1 \left(\xa\right) < \dotsb,
    $
    and if there are at least $k$ regeneration times,
    $\rp_k \left(\xa \right) \coloneqq x_{\rt_k \left(\xa\right)}$ is the
    $k$-th \emph{regeneration point} and
    $\rh_k \left(\xa\right) \coloneqq \abs{\rp_k \left(\xa\right)}$ is the
    $k$-th \emph{regeneration height}.
  \item for $u \in \prt\words$, the \emph{first hitting time} of
    the path $\xa$ to $u$ is
    \begin{align*}
      \tau_u \left(\xa\right) 
      \coloneqq 
      \inf \setof{ s \geq 0 }{ x_s = u},
    \end{align*}
    with the convention $\inf \emptyset = + \infty$.
\end{itemize}

A \emph{ray} is an infinite path $\rho$ such that
$\rho_0 = \root$ and for each $k \geq 0$, $\rho_{k+1}$ is a child of $\rho_k$.
In particular, for each $k \geq 0$, $\abs{\rho_k} = k$. Any transient path $\xa$
starting from $x_0 = \root$ defines a ray
\begin{equation*}
  \ray \left(\xa\right) 
  \coloneqq
  \left(
    \ep_0 \left(\xa\right), \ep_1 \left(\xa\right), \dotsc 
  \right).
\end{equation*}

Let 
$\infinitewords \coloneqq \left(\Ne\right)^{\Ne}$
the set of all infinite words. 
For $k \geq 0$, the $k$-th truncation of an infinite word $\xi$ is the finite word composed of
its $k$ first letters and is denoted $\xi_k$, with $\xi_0 \coloneqq \root$.
The mapping $\xi \mapsto \left(\xi_0, \xi_1, \xi_2, \dotsc \right)$ is a
bijection between infinite words and rays, therefore we will often abuse
notation and write $\xi$ for both the infinite word and the ray associated to it.
When a finite word $x$ is a truncation of an infinite word
$\xi$, we still say that $x$ is a prefix of $\xi$.
For two distinct infinite words $\xi$ and $\eta$, we may again
consider their greatest common prefix $\xi \gcp \eta \in \words$ and define
the \emph{natural distance} between each other by
\begin{equation}
\label{eq:dist1}
  \distrays \left(\xi, \eta \right) 
  =
  e^{- \abs{\xi \gcp \eta}}.
\end{equation}
This makes $\infinitewords$ into a complete, separable, ultrametric space.

\subsection{Trees and flows}

We represent our trees as subsets of the finite words on the alphabet $\Ne$.
A (rooted, planar, locally finite) tree $t$, \emph{without leaves},
is a subset of $\prt\words$ such that : 
\begin{itemize}
  \item $\root$ and $\prt\root$ are in $t$ ;
  \item for any $x \neq \prt\root$ in $t$,
    there exists a unique positive integer, denoted by $\numch (x)$ 
    and called the
    number of children of $x$ in $t$, such that for any $i \in \Ne$,
    $x i$ is in $t$ if and only if $i \leq \numch(x)$.
\end{itemize}
In this context, we call $\root$ the \emph{root} of $t$.
The tree $t$ is endowed with the undirected graph structure obtained by drawing
an edge between each vertex and its children.

We say that an infinite path $\xa$ in $\prt\words$
\emph{belongs} in $t$ if for any $k \geq 0$, $x_k$ is in $t$. The
\emph{boundary} of $t$ is the set
$\rays{t}$ of all rays that belong in it.
It is a compact subspace of $\infinitewords$.

A \emph{flow} on $t$ is a function $\theta$ from $t$ to $\intff{0}{1}$, such
that $\theta(\root) = 1$ and for any $x \in t$,
\begin{equation*}
  \theta (x) = \sum_{i=1}^{\nu_t(x)} \theta(xi).
\end{equation*}
Let $M$ a Borel probability on $\partial t$.
We may define a flow $\theta_M$ on $t$ by setting, for all 
$x \neq \prt\root$ in $t$, 
$\theta_M (x) = M \left(\setof{\xi \in \rays t}{x \prec \xi\right)}$. 
By a monotone class argument, the mapping $M \mapsto \theta_M$ is a
bijection and we will write $\theta$ for
both the flow on $t$ and the associated Borel probability measure on $\rays{t}$.

The (upper) \emph{Hausdorff dimension} of a flow $\theta$ 
on the tree $t$ is defined by
\begin{equation*}
  \hausd (\theta) 
  \coloneqq 
  \inf 
    \setof{\hausd (B) }{B \text{ Borel subset of } \rays{t}, \, \theta(B) = 1}.
\end{equation*}
When, for some $\alpha \geq 0$, $\theta$ satisfies
\[
  -\frac{1}{n} \log \left(\theta(\xi_n)\right)
  \xrightarrow[ n \to \infty ]{  }
  \alpha
  \quad\quad
  \text{for $\theta$-almost every $\xi$,}
\]
one says that $\theta$ is \emph{exact-dimensional} and many alternative
definitions of its dimension coincide (see
\cite[chapter~10]{falconer1997techniques} and \cite{mattila2000dimension})
. In particular, $\hausd\theta = \alpha$.
We say that the \emph{dimension drop phenomenon occurs for $\theta$} when
$\hausd \theta < \hausd \rays t$.

\subsection{Weighted trees}
\label{sub:weighted_trees}

A \emph{weighted tree} is a tree $t \in \noleaf$ together with a function $\wt$
from $t \setminus \left\{ \root, \prt\root\right \}$ to $ \intoo{0}{\infty}$.
For $x$ in $t \setminus\{ \prt\root, \root \}$, we call $\wt(x)$ the
\emph{weight}
of $x$ in $t$.

We will only work with weighted trees but to lighten notations, we will write
$t$ when we should write $ \left(t, \wt\right)$.
We still, however, write $x \in t$ when we mean that a word $x$ is a vertex of
the weighted tree $t$.

We define the (local) distance between two weighted trees $t$ and $t'$ by
\begin{equation*}
  d_{w} \left(t, t'\right) 
  = 
  \sum_{r \geq 0}^{} 2^{-r-1} \delta^{(r)}_m \left(t, t'\right),
\end{equation*}
where $\delta^{(r)}_w$ is defined by
\begin{equation*}
  \delta^{(r)}_w \left(t, t'\right)
  =
  \left\lbrace
  \begin{array}{l}
    1 \text{ if $t$ and $t'$ (without their weights) do not agree up to height $r$}; \\
    \min \left( 1, 
      \sup\setof{\abs{\wt(x) - \wt[t'](x)}}%
        {x \in t, \, 1 \leq \abs{x} \leq r}
    \right) 
    \text{ otherwise.}
  \end{array}
  \right.
\end{equation*}
We denote 
$\weightedtrees$ the metric space of all weighted trees. It is a Polish space.
For a weighted tree $t$ and a vertex $x \in t$, we denote
\begin{equation*}
  t [x] \coloneqq \setof{u \in \prt\words}{x u \in t},
\end{equation*}
the reindexed subtree starting from $x$ with weights
\begin{equation*}
  \wt[ { t[x] }] (y) 
  \coloneqq
  \wt \left(x y\right), 
  \quad
  \forall y \in t[x] \setminus \left\{ \root, \prt\root \right\}.
\end{equation*}

The tree $t$ \emph{pruned} at a vertex $x \neq \prt\root$
in $t$ is the subset of $\prt\words$ defined by
\begin{equation*}
  t^{\leq x} 
  \coloneqq 
  \setof{ y \in t }{ x \nprec y }.
\end{equation*}
Notice that $x$ is in $t^{\leq x}$.
We will write $ t^{\leq x}$ when we mean $t^{\leq x}$ together with the
restriction of $\wt$ to $t^{\leq x}$.
The set of all pruned (weighted) trees is
\begin{equation*}
  \weightedchopped \coloneqq
  \setof{ t^{\leq x}}{t \in \weightedtrees, \, x \in t}.
\end{equation*}
We equip it with a local distance similar to $d_{w}$.

For two weighted trees $t$ and $t'$, and $x \neq \prt\root $ in $t$,
we define the glued weighted tree $t^{\leq x} \gluetrees t'$ as the tree 
\begin{equation*}
  t^{\leq x} \gluetrees t' \coloneqq
  t^{\leq x} \cup \setof{xy}{y \in t'\setminus\{\prt\root, \root\}}
\end{equation*}
together with the weights :
\begin{equation*}
  \wt[t^{\leq x} \gluetrees t'] (z)
  = \begin{cases}
  \wt (z) & \text{if $x \nprec z$} ; \\
  \wt[t'] (x^{-1}z) & \text{otherwise}.
\end{cases}
\end{equation*}
Notice that in particular the weight of $x$ in
$t^{\leq x}\gluetrees t'$ is still $\wt(x)$.
\subsection{Flow rules on weighted trees}
A (positive and consistent) flow rule on weighted trees
is a measurable mapping $t \mapsto \Theta_t$ from a Borel subset $\mathcal{I}$ of
$\weightedtrees$ to the set of functions $\intff{0}{1}^\words$ such
that:
\begin{itemize}
  \item for any weighted tree $t$ in $\mathcal{I}$, $\Theta_t$ is a flow on $t$;
  \item for any $x$ in $t$, $t[x] \in \mathcal{I}$ and $\Theta_t(x) > 0$;
  \item for any $xy$ in $t$,
    \begin{equation*}
      \Theta_t (xy) = \Theta_t(x) \Theta_{t[x]} (y).
    \end{equation*}
\end{itemize}
Notice that in our context, $\Theta_t$ might depend on the weights of the
weighted tree $t$.
As a simple first example, let us define the flow rule $\visw$ as in
\cite{liu1997two} (it is denoted there $\overline{\nu}$; when the weights are
all equal, we recover the flow rule $\vis$ of \cite{LPP95}).
For a weighted tree $t$ and $x$ in $t\setminus\{\prt\root\}$, we let
\[
  \visw_t(x) \coloneqq
  \prod_{\root \prec u \preceq x}
  \frac{\wt(y)}{
  \sum_{i=1}^{\numch(\prt y)} \wt(\prt y i)}.
\]
In other words, this is the law of the trajectory of a non-backtracking random
walk starting from $\root$, which randomly chooses a child of its current
position with probability proportional to its weight.
See \cite[Theorem~7]{liu1997two} for the exact dimensionality and
the Hausdorff dimension of $\visw_T$, when
$T$ is a random weighted Galton-Watson tree.

For other examples of flow rule, see $\harm$ (the \emph{harmonic flow rule}) in the next subsection and
$\unif$ (the \emph{limit uniform measure}) in \cite[Section~6]{LPP95} and in
\cref{sec:invariant_harm}.
\subsection{Random walks on weighted trees}

Let $t$ a weighted tree. We associate to $t$ a transition matrix : for all $x \neq
\prt\root$ in $t$ and all $y$ in $t$
\begin{equation}
  \label{eq:defpt}
  \Pq^t \left(x, y\right)
  =
  \begin{cases}
    1 / \left(1 + \sum_{j=1}^{\numch(x)} \wt(xj) \right)
    & \text{if } y = \prt x ; \\
    \wt (xi) / \left(1 + \sum_{i=j}^{\nu_t(x)} \wt(xj) \right)
    & \text{if } y = xi, \text{ for } 1 \leq i \leq \numch(x) ; \\
    0 & \text{otherwise.}
  \end{cases}
\end{equation}
The walk is reflected at $\prt\root$, that is, $\Pq^t (\prt\root, \root) = 1$.
For $x$ in $t$, we denote
$\Pq_x^t$ the probability measure under which the random path 
$\Xa = \left(X_0, X_1, \dotsc \right)$ in $t$ is a
Markov chain starting from $x$ with transition matrix $\Pq^t$.
The associated expectation is denoted $\Eq^t_x$.
Since we will later consider random weighted tree, $\Pq^t_x$ and $\Eq^t_x$
will often be referred to as the ``quenched'' probabilities and expectations.

We say that a weighted tree $t$ is \emph{everywhere transient} if, 
for all $x \neq \prt\root$,
the random path $\Xa$ in $t[x]$ is $\Pq_\root^{t[x]}$-almost surely transient.
When a weighted tree $t$ is everywhere transient, the harmonic measure $\harm_t$ on
its boundary $\rays t$ is the law of
$\Xi \coloneqq \ray \left(\Xa\right)$.
The mapping $t \mapsto \harm_t$ on the set of everywhere transient weighted
trees is then a (positive and consistent) flow rule.

\subsection{Weighted Galton-Watson trees}
\label{sub:weighted_gw}
We consider a reproduction law $\pp = \left(p_k\right)_{k \geq 0}$, that is
a sequence of non-negative real numbers such that
$\sum_{k=0}^{\infty} p_k = 1$. We assume throughout this work
that $p_0 = 0$ and $m \coloneqq \sum_{k \geq 1}^{} p_k k < \infty$.
Under some probability $\Pt$, let $(N, A)$ a random element
in $\Ne\times\bigcup_{k\geq 1}\intoo{0}{\infty}^k$ such that the marginal
law of $N$ is $\pp$ and, for
any $k\geq 1$, if $N = k$, then $A$ lies in $\intoo{0}{\infty}^k$.

The law of the random weighted Galton-Watson tree
$T$ under the probability $\Pt$ is
defined recursively in the following way:
\begin{itemize}
  \item The joint law of 
    $ \left(\nu_T(\root), \wt[T](1), \wt[T](2), \dotsc, \wt[T](\nu_T(\root))\right) $ 
    is the law of $(N,A)$ ;
  \item the reindexed subtrees 
    $T[1]$, \dots,
    $T[\nu_T(\root)]$ are independant and have the same law as $T$.
\end{itemize}
We call the resulting tree a weighted Galton-Watson tree 
 and denote its law $\GW$.
Note that, if we forget the weights, we recover a classical 
Galton-Watson tree of reproduction law $\pp$.
The branching property is still valid in this setting of weighted trees.
More precisely, let
$k \in \Ne$, $ B_1, B_2, \dotsc, B_k $ Borel sets of $ \intoo{0}{\infty}$
and $ \mathcal{B}_1, \mathcal{B}_2, \dotsc, \mathcal{B}_k $ Borel sets of 
$ \weightedtrees $, 
\begin{multline}
  \label{eq:branching1}
  \Pt \left(\nu_T (\root) = k, \, \wt[T] (1) \in B_1, \dots, \wt[T] (k) \in B_k,
  \,
  T[1] \in \mathcal{B}_1 , \dots, T[k] \in \mathcal{B}_k \right) \\
  =
  p_k \Pt \prcd{A \in B_1 \times \dotsm \times B_k}{N = k}
    \prod_{i=1}^{k} \GW \left( \mathcal{B}_i \right).
\end{multline}
Let $f$ a Borel, positive or bounded, function on $\weightedchopped$
and $g$ a Borel, positive or bounded, function
on the space $\weightedtrees$. A consequence of the previous identity is the
following version of the branching property that we will use constantly
throughout this work.
\begin{equation}\label{eq:branching2}
  \Et \left[\indic{x \in T} f \left(T^{\leq x}\right)
  g \left(T[x]\right)\right]
  =
  \Et \left[\indic{x \in T} f\left(T^{\leq x}\right)\right] 
  \Et \left[g\left(T\right)\right].
\end{equation}

We associate to $T$ a transition matrix $\Pq^T$ as in \eqref{eq:defpt} and the associated random walk
$\Xa$.
\begin{exe}
  Let $\lambda > 0$.
  If the weights are all constant and equal to $\lambda^{-1}$, the model is called
  the $\lambda$-biased random walk on a Galton-Watson tree. Given the
  Galton-Watson tree $T$, the transition
  matrix $\Pq^T$ is defined by:
\begin{equation*}
  \Pq^T \left(x, y\right)
  =
  \begin{cases}
    \lambda / \left( \lambda + \numch[T](x) \right)
    & \text{if } y = \prt x ; \\
    1 / \left(\lambda + \nu_T(x) \right)
    & \text{if } y = xi, \text{ for } 1 \leq i \leq \numch[T](x) ; \\
    0 & \text{otherwise.}
  \end{cases}
\end{equation*}
The associated random walk is almost surely transient if and only if 
$\lambda < m$ (see \cite{lyons_1992}).
The dimension drop for the harmonic measure is established in \cite{LPP95} for
$\lambda = 1$ (simple random walk) and in \cite{LPP_biased} for 
$0 < \lambda < m$.
\end{exe}

In \cite{lyons_pemantle_rwre} we have a transience criterion for the random walk
$ \Xa $ on $T$ with (quenched) transition matrix $\Pq^T$, when the weights are
i.i.d. It is generalized for our setting in \cite[Theorem~1.1]{faraud_2011}. 
One can see that the integrability assumptions are not needed for the proof of
the transient case.
\begin{fact}[{\cite[theorem~1.1]{faraud_2011}}]
  If $\min_{\alpha \in \intff{0}{1}}\Et \left[\sum_{i=1}^{\nu_T(\root)}
  \wt[T](i)^\alpha\right] > 1$, then for $\GW$-almost every weighted tree $t$, the
  random walk defined by $\Pq^t$ is transient.
\end{fact}
We will assume throughout this work that we are in this regime.
\subsection{Basic facts of ergodic theory}
\label{sub:basicet}
We recall here some definitions and basic properties that are used in this paper. The
notations of this section are local to this section.
\begin{defi}
  Let $\left(X, \mathcal{F}_X\right)$ and $\left(Y, \mathcal{F}_Y\right)$ two
  measurable spaces and $S_X : X \to X$, $S_Y : Y \to Y$ two measurable
  transformations.
  A \emph{semi-conjugacy} between $\left(X, \mathcal{F}_X, S_X\right)$
  and $\left(Y, \mathcal{F}_Y, S_Y\right)$ is a surjective measurable mapping
  $h : X \twoheadrightarrow Y$ such that $h \circ S_X = S_Y \circ h$.
  
  One says that $h$ is a \emph{conjugacy} between 
  $\left(X, \mathcal{F}_X, S_X\right)$ and 
  $\left(Y, \mathcal{F}_Y, S_Y\right)$
  if, in addition, the semi-conjugacy 
  $h$ is also injective.
\end{defi}
The following well-known 
fact can be checked very directly, so we omit the proof.
\begin{fact}
  Let $\left(X, \mathcal{F}_X, S_X\right)$ and 
  $\left(Y, \mathcal{F}_Y, S_Y\right)$ two measurable spaces endowed with a
  measurable transformation.
  Let $h : X \twoheadrightarrow Y$ a semi-conjugacy and $\mu_X$ a probability measure on $\mathcal{F}_X$.
  Then, if the system $\left(X, \mathcal{F}_X, S_X, \mu_X\right)$ is
  measure-preserving (resp. ergodic, mixing), so is
  $\left(Y, \mathcal{F}_Y, S_Y, \mu_X\circ h^{-1}\right)$.
\end{fact}
%\begin{proof}
%  Assume that $\left(X, \mathcal{F}_X, S_X, \mu_X\right)$ is measure-preserving.
%  Denote $\mu_Y \coloneqq \mu_x \circ h^{-1}$.
%  Let $A$ in $\mathcal{F}_Y$. By definition of $\mu_Y$,
%  \[
%    \mu_Y \left(S_Y^{-1} (A)\right)
%    = \mu_X \left(h^{-1}\circ S_Y^{-1} (A) \right)
%    = \mu_X \left(S_X^{-1} \left(h^{-1} (A)\right)\right)
%    = \mu_X \left(h^{-1}(A)\right) = \mu_Y(A),
%  \]
%  which proves the first point.
%
%  Now assume that $ \left(X, \mathcal{F}_X, S_X, \mu_X \right)$ is ergodic.
%  Let $A$ in $\mathcal{F}_Y$ such that
%  $S_Y^{-1} (A) = A$. Then,
%  $
%    h^{-1} \circ S_Y^{-1} (A) = h^{-1} (A)
%    $, so $S_X^{-1} \left(h^{-1} (A) \right) = h^{-1} (A)$ and
%    $\mu_Y (A) = \mu_X( h^{-1}(A) )= 0 \text{ or } 1$, which proves the second
%    point.
%
%  Finally, assume that $\left(X, \mathcal{F}_X, S_X, \mu_X \right)$ is
%  mixing and
%  let $A$ and $B$ in $\mathcal{F}_Y$. By the fact that for any $n \geq 1$,
%  $h^{-1} \circ S_Y^{-n} = S_X^{-n} \circ h^{-1}$, we have
%  \[
%    \mu_Y \left(S_Y^{-n} (A) \cap B\right)
%    = \mu_X \left(S_X^{-n} \left(h^{-1}(A)\right) \cap h^{-1}(B)\right)
%    \xrightarrow[ n \to \infty ]{  }
%    \mu_X \left(h^{-1} (A) \right) \mu_X \left(h^{-1}(B)\right),
%  \]
%  which concludes the proof.
%\end{proof}
\begin{defi}
  Let  $\left(X, \mathcal{F}, S, \mu \right)$ a measure-preserving system (with
  $\mu (X) = 1$)
  and $A$ in $\mathcal{F}$ such that $\mu (A) > 0$. For $x$ in $X$, let
  \[
    n_A (x) =
    \inf \setof{ k\geq 1}{S^k (x) \in A},
  \]
  with the convention $\inf \emptyset \coloneqq +\infty$.
  For $B$ in $\mathcal{F}$, let
  $\mu_A (B) = \mu (A \cap B) / \mu(A)$ and for $x$ in $X$, let
  $S_A (x) \coloneqq S^{n_A(x)} (x)$ if $n_A(x)$ is finite and (say)
  $S_A (x) \coloneqq x$ if $n_A(x) = \infty$.
  The \emph{induced system} on $A$ is defined as
  $\left(A, \mathcal{F}\cap A, S_A, \mu_A\right)$.
\end{defi}
\begin{lem}
  With the notations and assumptions of the previous definition, the system
  $\left(A, \mathcal{F}\cap A, S_A, \mu_A\right)$ is measure-preserving.
  Moreover,
  we have that the whole system $\left(X, \mathcal{F}, S, \mu \right)$ is ergodic if and only if 
  $\mu \left( \bigcup_{k\geq 1}^{} S^{-k}(A)\right) = 1$ and
  $\left(A, \mathcal{F}\cap A, S_A, \mu_A\right)$ is ergodic.
\end{lem}
We provide a short proof of the ``if'' part, since we did not find it in the
litterature. For the other assertions, see for instance 
\cite[Lemma~2.43]{einsiedler_ward}.
\begin{proof}
  For $k$ in $\Ne\cup\{\infty\}$, let
  $A_k \coloneqq \setof{x \in A}{n_A(x) = k}$.
  Assume that \[\mu \left( \bigcup_{k\geq 1}^{} S^{-k}(A)\right) = 1\] and
  $\left(A, \mathcal{F}\cap A, S_A, \mu_A\right)$ is ergodic.
  Let $C$ in $\mathcal{F}$ such that $S^{-1}(C) = C$.
  We prove that $C \cap A$ is $S_A$-invariant.
  Indeed,
  \begin{align*}
    S_A^{-1}(C \cap A)
    &=
    S_A^{-1}(C \cap A) 
    \cap A_{\infty}\sqcup  \bigsqcup_{k \geq 1} S^{-k}(C\cap A) \cap A_k
    \\
    &=
    C \cap A_{\infty}\sqcup  \bigsqcup_{k \geq 1} C\cap S^{-k}(A) \cap A_k
    \\
    &= 
    C \cap A_{\infty}\sqcup  \bigsqcup_{k \geq 1} C\cap A_k
    = C \cap A.
  \end{align*}
  Thus, $\mu \left(C \cap A\right)$ equals $0$ or $\mu(A)$. If it is $0$,
  then
  \[
    \mu (C)
    = \mu \left(C \cap \bigcup_{k\geq 1} S^{-k}(A)\right)
    \leq \sum_{k \geq 1}^{} \mu \left(C \cap S^{-k}(A)\right) = 0,
  \]
  since, for any $k \geq 1$,
  \[
    \mu \left(C \cap S^{-k}(A)\right)
    = \mu \left(S^{-k}(C) \cap S^{-k}(A)\right) = \mu (C \cap A).
  \]
  If $\mu (C\cap A) = \mu(A)$, we reason on the complement $\compl{C}$ of $C$,
  which is still invariant by $S$ and satisfies
  $\mu (\compl{C}\cap A) = 0$.
\end{proof}
\section{Regeneration Times}
\label{sec:regeneration}
Let $\weightedtreeswithpath$ be the space of all trees $t$ in $\weightedtrees$
with a distinguished transient path $\xa$ starting from the root.
On $\weightedtreeswithpath$, we define the distance $d_{w,p}$ by
\begin{equation*}
  d_{w,p} \left( \left(t, \xa \right), \left(t', \xa'\right)\right) = 
  \sum_{r \geq 0}^{} 2^{-r-1} 
  \delta_{w,p}^{(r)} \left( \left(t, \xa \right), 
  \left(t', \xa'\right)\right),
\end{equation*}
where 
$\delta_{w,p}^{(r)} \left( \left(t, \xa\right), \left(t', \xa'\right)\right) 
  = 1$
if the vertices of $t$ and of $t'$ do not agree up to height $r$
or if the paths $\xa$ and $\xa'$ do not coincide before the first
time they reach height $r+1$.
Otherwise, 
$\delta_{w,p}^{(r)} \left( \left(t, \xa\right), \left(t', \xa' \right)\right) 
  = \delta_w^{(r)} \left(t, t'\right)$.
The metric space $\weightedtreeswithpath$ is again Polish.

Following \cite[proof of Proposition~3.4]{LPP_biased}, for any $s$ in $\ft \left(\xa\right)$, we consider the tree and the path before
time $s$ : 
\begin{equation*}
  \phi_s \left(t, \xa \right)
  \coloneqq \left(t^{\leq x_s}, \left(x_i\right)_{0 \leq i \leq s}\right).
\end{equation*}
Likewise if $s$ is in $\et \left(\xa\right)$, the reindexed tree and path after
time $s$ is
\begin{equation*}
  \psi_s \left(t, \xa \right)
  = \left(t \left[x_s\right] , \xa \left[s\right] \right),
\end{equation*}
where 
\begin{equation*}
  \xa \left[s \right] \coloneqq \left(x_s^{-1} x_{s + k}\right)_{k \geq 0}.
\end{equation*}
From the definition of fresh times and exit times,
each path is in the corresponding tree or pruned tree.

Now, let $T$ a random weighted tree of law $\GW$ and, conditionally on $T$, let
$\Xa$ a trajectory of the random walk with transition matrix $\Pq^T$,
starting from $\root$.
We denote $\Pa$ the ``annealed'' probability, that is, the probability
associated to the expectation $\Ea$ defined
by
\begin{equation*}
  \Ea \left[f\left(T, \Xa\right) \right] 
  \coloneqq
  \Et \left[ \Eq^T_\root \left[f\left(T, \Xa\right)\right]  \right],
\end{equation*}
for all suitable measurable functions $f$ on $\weightedtreeswithpath$.

For short, we write $\ft$ for $\ft \left( \Xa\right)$,
$\fp$ for
$\fp \left(\Xa \right)$, etc, and $\psi_s$, $\phi_s$ for
$\psi_s \left(T, \Xa\right)$ and $\phi_s \left(T, \Xa\right)$.

\begin{lem}\label{lem:rtbranching}
  Let $s$ in $\Ne$, and $f$ and $g$ measurable and non-negative (or bounded)
  functions, respectively on $\weightedtreeswithpath^{\leq}$ and
  $\weightedtreeswithpath$.
  Then
  \begin{equation*}
    \E \left[\indic{s \in \rt} 
      f \left(\phi_s \right)
      g \left(\psi_s \right)
    \right] 
    =
    \E \left[\indic{s \in \ft}
    f \left(\phi_s\right) \right]
    \E \left[g \left(T, \Xa \right) \indic{\tau_{\prt\root} = \infty}
    \right]. 
  \end{equation*}
\end{lem}
\begin{proof}
  We first decompose the expectation according to the value of $X_s$.
  \begin{multline*}
    \E \left[\indic{s \in \rt} 
      f \left(\phi_s \right)
      g \left(\psi_s \right)
    \right] 
    \\
    =
    \sum_{x \in \words}^{}
    \Et \left[
      \indic{x \in T}
      \Eq_\root^T \left[
        \indic{X_s = x, \,
          s \in \ft}
        f \left(T^{\leq x}, \left(X_i\right)_{0\leq i \leq s}\right)
        \indic{s \in \et}
      g \left(T [x] , \left(x^{-1} X_{s+k}\right)_{k \geq 0} \right)
    \right] 
    \right].
  \end{multline*}
  By the Markov property at time $s$, for any fixed $x$ in $T$,
  the quenched expectation can be rewritten
  \begin{multline*}
    \Eq_\root^T \left[
        \indic{X_s = x, \,
          s \in \ft}
        f \left(T^{\leq x}, \left(X_i\right)_{0\leq i \leq s}\right)
        \indic{s \in \et}
      g \left(T [x] , \left(x^{-1} X_{s+k}\right)_{k \geq 0} \right)
    \right] 
    \\
    =
      \Eq_\root^T \left[
        \indic{X_s = x, \,
          s \in \ft}
        f \left(T^{\leq x}, \left(X_i\right)_{0\leq i \leq s}\right)
      \right]
      \Eq_x^T \left[
        \indic{\tau_{\prt x} = \infty}
      g \left(T [x] , \left(x^{-1} X_{k}\right)_{k \geq 0} \right)
    \right]. 
  \end{multline*}
  Now, the first quenched expectation is
  only a function of the weighted tree $T^{\leq x}$ while the second is only a
  function of $T[x]$. So we can use the branching property
  \eqref{eq:branching2} and sum over $x$ in $\words$ to get the result.
\end{proof}

\begin{defi}
  The \emph{conductance} of a weighted tree $t$ is
  \begin{equation*}
    \beta (t) \coloneqq \Pq^t_\root \left(\tau_{\prt\root} = \infty \right).
  \end{equation*}
\end{defi}
From the theory of discrete-time Markov chains, one can check that the tree $t$ is transient if and only if $\beta(t) > 0$.

\begin{lem}
  For $\GW$-almost every weighted tree $t$, for $\Pq^t_\root$-almost every path
  $\xa$, the set $\rt \left(\xa\right)$ is infinite.
\end{lem}
\begin{proof}
  This proof is very similar to \cite[Lemma~3.3]{LPP_biased}.
  For $k \geq 1$, let $\mathcal{F}_k$ the $\sigma$-algebra on
  $\weightedtreeswithpath$ generated by $X_0, X_1, \dotsc, X_k$ and
  $\mathcal{F}_{\infty}$ the $\sigma$-algebra generated by the whole path $\Xa$.
    For $N$ in $\N$, let $\ft^{(N)}$ the first fresh time after (or at) time $N$.
  Then, 
  \begin{equation*}
    \P \brcd{ \bigcup_{s \geq N}^{} \left\{s \in \rt\right\} }{ \mathcal{F}_N }
    \geq
    \P \brcd{ \ft^{(N)} \in \rt }{ \mathcal{F}_N } \\
    =
    \sum_{s \geq N}^{}
    \E \brcd
    { \indic{\ft^{(N)} = s} \indic{s \in \rt} }{ \mathcal{F}_N }.
  \end{equation*}
    Thus we can use \cref{lem:rtbranching} to get
    \begin{equation*}
    \P \brcd{ \bigcup_{s \geq N}^{} \left\{s \in \rt\right\} }{ \mathcal{F}_N }
    \geq
    \sum_{s \geq N}^{}
    \E \brcd{ \indic{\ft^{(N)} = s } }{ \mathcal{F}_N }
    \E \left[ \indic{\tau_{\prt\root} = \infty} \right]
    =
    \Et \left[\beta \left(T\right)\right] > 0,
  \end{equation*}
  By regular martingale convergence theorem and the fact that for any $N$ in
  $\N$, the event
  $ \bigcup_{s \geq N} \{ s \in \rt \} $ is in $\mathcal{F}_\infty$, we have
  almost surely,
  \begin{align*}
    \indi{ \bigcup_{s \geq N} \{s \in \rt \} }
    &= \lim_{k \to \infty}
    \P \brcd{ \bigcup_{s \geq N}^{} \left\{s \in \rt\right\} }{
    \mathcal{F}_{N+k} }
    \\
    &\geq 
    \P \brcd{ \bigcup_{s \geq N + k }^{} \left\{s \in \rt\right\} }{
    \mathcal{F}_{N+k} }
    \geq \Et \left[\beta (T) \right] > 0.
  \end{align*}
  Hence, $\indi{\bigcup_{s \geq N} \{ s \in \rt \} } = 1$, almost surely.
\end{proof}

We will now work on the space of weighted trees with transient paths that have
infinitely many regeneration times. We still denote it
$\weightedtreeswithpath$ in order not to add another notation.

%We write, for $s \geq 0$, $\Psi_s \coloneqq \Psi_s \left(T, \Xa \right)$
%and $\Phi_s \coloneqq \Psi_s \left(T, \Xa\right)$ for short.
\begin{prop}
  \label{prop:slab}
  Let $f$ and $g$ measurable (or bounded) functions. For any $n \geq 1$,
  \begin{equation}\label{eq:atregen}
    \begin{split}
      \E \left[ 
        f\left(\Phi_{\rt_n}\right)
        g \left(\Psi_{\rt_n} \right)
      \right] 
      &=\E \left[
        f \left(\Phi_{\rt_n}\right)
      \right] 
      \E \brcd{
      g \left(T, \Xa \right)}{ \tau_{\prt\root} = \infty}
      \\
      &= \E \left[f \left(\Phi_{\rt_n} \right)\right]
      \E \left[ g \left(\Psi_{\rt_n}\right)\right].
    \end{split}
  \end{equation}
\end{prop}

\begin{proof}
  Again, this proof is similar to \cite[p.~255]{LPP_biased}.
  For $1 \leq n \leq s$, let $C_n^s$ the event that exactly $n$ edges have been
  crossed exactly one time before time $s$. Reasoning on the value of the $n$-th
  regeneration time, we first get
\begin{equation*}
  \E \left[ f\left(\Phi_{\rt_n}\right) g \left(\Psi_{\rt_n}\right)\right] 
  =
  \sum_{s \geq n}^{}
  \E \left[ \indic{\rt_n = s} f \left(\Phi_s\right) g 
    \left(\Psi_s\right)\right] 
  =
  \sum_{s \geq n}^{}
  \E \left[ \indic{s \in \rt} \indi{C_{n-1}^{s-1}} f \left(\Phi_s\right) 
   g \left(\Psi_s\right)\right].
\end{equation*}
On the event $\{ s \in \rt \}$, the indicator $\indi{C_{n-1}^{s-1}}$ is a
function of $\Phi_s$, thus, using \cref{lem:rtbranching}, we obtain
\begin{align*}
  \E \left[ f\left(\Phi_{\rt_n}\right) g \left(\Psi_{\rt_n}\right)\right] 
  &=
  \E \left[\indic{\tau_{\prt\root} = \infty} g \left(T, \Xa\right)\right] 
  \sum_{s \geq n}^{}
  \E \left[\indic{s \in \ft} \indi{C_{n-1}^{s-1}}
    f \left( \Phi_s \right)
  \right] 
  \\
  &= \E \brcd{ g \left(T, \Xa\right)}{\tau_{\prt\root} = \infty}
  \sum_{s \geq n}^{} \E \left[\indic{\tau_{\prt\root} = \infty}\right] 
  \E \left[\indic{s \in \ft} \indi{C_{n-1}^{s-1}}
    f \left( \Phi_s \right)
  \right]. 
\end{align*}
  Using \cref{lem:rtbranching} the other way around,
\begin{align*}
  \E \left[ f \left( \Phi_{\rt_n} \right) g \left( \Psi_{\rt_n} \right) \right] 
  &= \E \brcd{ g \left(T, \Xa\right) }{ \tau_{\prt\root} = \infty }
  \sum_{s \geq n}^{} 
  \E \left[ \indic{s\in \et} \indic{s \in \ft} \indi{C_{n-1}^{s-1}}
    f \left( \Phi_s \right)
  \right] 
  \\
  &= 
    \E \brcd{
    g \left( T, \Xa \right)}{ \tau_{\prt\root} = \infty}
    \E \left[
      f \left(\Phi_{\rt_n}\right)
    \right].
\end{align*}
Finally, taking $f$ constant equal to one yields the last equality.
\end{proof}
We define on $\weightedtreeswithpath$ the shift at exit times
\begin{equation*}
  \Se : \left(t,\xa \right) \mapsto 
  \Psi_{\et_1 \left(\xa\right)} \left( t, \xa \right)
= \left( t\left[\ep_1\right] , \xa \left[\et_1\right] \right),
\end{equation*}
and the shift at regeneration times
\begin{equation*}
  \Sr :
  \left(t,\xa \right) \mapsto \Psi_{\rt_1 \left(\xa\right)} \left( t, \xa \right)
= \left(t \left[\rp_1\right] , \xa \left[\rt_1\right] \right).
\end{equation*}
For $k \geq 1$, let
\begin{equation*} \Se^k \coloneqq \underbrace{
\Se \circ \dotsb \circ \Se}_{k \text{ times}}. \end{equation*}
Since regeneration times are exit times, for any $(t, \xa)$ in
$\weightedtreeswithpath$,
\begin{equation*}
  \Sr \left(t, \xa\right) =
  \Se^{\rh_1 \left(\xa\right)} \left(t, \xa \right).
\end{equation*}
\begin{cor}
  The law $\mureg$ of $\left(T, \Xa\right)$ on
  $\weightedtreeswithpath$, under the probability measure
  $\P^{*} \coloneqq \P \brcd{ \cdot }{\tau_{\prt\root} = \infty}$ 
  is invariant and mixing with respect to the shift $\Sr$.
\end{cor}
\begin{proof} For the invariance, take $f$ constant equal to one
  in \eqref{eq:atregen}.

  Now let $f$ and $g$ non-negative measurable functions on $\weightedtreeswithpath$.
  By a monotone class argument, we may assume that $g$ only depends on the $N$
  first generations of the weighted
  tree and on the path until it escapes these generations for the first time.

  Since the $N$-th regeneration point is at least of height $N$, we get, using
  \eqref{eq:atregen}, for all $k \geq N$,
  \begin{align*}
    \E^{*} \left[f \circ \Sr^k \left(T, \Xa \right) g \left(T,\Xa \right)\right]
    &= \E^{*} \left[f \left(T [\rp_k], \Xa [\rt_k]\right) 
      g \left(T, \Xa \right)
    \right] 
    \\
    &= \E^{*} \left[f \left(T, \Xa\right)  \right] 
    \E^{*} \left[ g\left(T, \Xa\right) \right],
  \end{align*}
  thus the system is mixing.\end{proof}

\section{Tower Construction of an Invariant Measure for the Shift at Exit Times}
\label{sec:invariant_exit}

We now build a Rokhlin tower over the system
$ \left(\weightedtreeswithpath, \Sr, \mu_r\right)$ 
in order to obtain a probability measure which is
invariant with respect to the shift $\Se$. This is a classical and general
construction but we provide details in our specific case for the reader's
convenience.

For any $i \geq 1$, let
\begin{equation*}
  E_i \coloneqq \setof{ \left(t, \xa\right) \in \weightedtreeswithpath}{
  \rh_1 \left(\xa\right) \geq i}.
\end{equation*}
We then have
\begin{equation*}
  \weightedtreeswithpath = E_1 \supset E_2 \supset \dotsb .
\end{equation*}
For $i \geq 1$, let 
$\tilde{E}_i \coloneqq E_i \times \{ i \}$ and
$\tilde{E} \coloneqq \bigsqcup_{i \geq 1} \tilde{E}_i$.
Let
$\phi_i : E_i \to \tilde{E}_i$ the natural bijection.
We define the measure
$\tilde{\mureg}^0$ by : for any measurable $\tilde{A}$ in $\tilde{E}$,
\begin{equation*}
  \tilde{\mureg}^{0} \left( \tilde{A} \right) 
  \coloneqq
  \sum_{i \geq 1}^{} \mureg 
  \left( 
    \phi_i^{-1} \left( \tilde{A} \cap \tilde{E}_i \right)
  \right).
\end{equation*}
The total mass of $\tilde{\mureg}^0$ is
\begin{equation*}
  \tilde{\mureg}^0 \left(\tilde{E}\right)
  = \sum_{i \geq 1}^{} \P^{*} \left(\rh_1 \geq i \right)
  = \E^{*} \left[\rh_1 \right].
\end{equation*}
\begin{lem}\label{lem:exprh}
  The expectation $\Ea^{*} \left[\rh_1\right] $ is finite and equals
  $\Et \left[\beta \left(T\right)\right]^{-1}$.
\end{lem}
We will prove this lemma later.
We now write
  $\tilde{\mureg} \coloneqq \tilde{\mureg}^0 / 
  \tilde{\mureg}^0 \left(\tilde{E}\right)$.
  We define the shift $\St$ on $\tilde{E}$ by :
\begin{equation}
  \St \left(t, \xa, i \right) \coloneqq
  \begin{cases}
    \left(t, \xa, i+1\right) & \text{if $\rh_1(\xa) \geq i+1$} ;
    \\
    \left(\Sr \left(t, \xa\right), 1 \right) &
      \text{if $\rh_1 \left(\xa\right) = i$}.
  \end{cases}
\end{equation}
\begin{lem}
  The measure $\tilde{\mureg}$ is invariant and ergodic with respect to the shift $\St$.
\end{lem}
\begin{proof}
  Let $f : \tilde{E} \to \R_+ $ a measurable function.
  \begin{align*}
    &\int f \circ \St
    \left(t, \xa, i \right) \dd \tilde \mureg \left(t, \xa, i \right)
    = \sum_{j \geq 1}^{} 
    \int_{\tilde{E}_j}
    f \circ \St \left(t, \xa, i \right) 
    \dd \tilde \mureg \left(t, \xa, i \right)
    \\
    &= \E^{*}\left[\rh_1\right]^{-1}\sum_{j \geq 1}^{} 
    \int_{E_j}
    f \circ \St 
    \left(t, \xa, j \right) \dd  \mureg \left(t, \xa\right)
    \\
    &= \E^{*}\left[\rh_1\right]^{-1} \left(\sum_{j \geq 1}^{} 
    \int_{E_j \setminus E_{j+1}}
    f  
    \left(\Sr(t, \xa), 1 \right) \dd  \mureg \left(t, \xa\right)
    +
    \int_{E_{j+1}}
    f
    \left(t, \xa, j+1 \right) \dd  \mureg \left(t, \xa\right)
  \right)
    \\
    &= \E^{*}\left[\rh_1\right]^{-1}
    \left(\int_{E_1} f \left(\Sr \left(t, \xa\right), 1 \right) 
    \dd \mureg \left(t, \xa \right)
    + \int_{\tilde{E} \setminus \tilde{E}_1}
    f \left(t, \xa, i \right) \dd \tilde \mureg \left(t, \xa, i \right)
  \right).
  \end{align*}
  The fact that $\Sr$ is invariant with respect to $\mureg$ concludes the proof
  of the invariance.

  For the ergodicity, we remark that, by construction,
  \begin{equation*}
    \bigcup_{k = 1}^{\infty} \St^{-k} \left(\tilde{E}_1\right)
    = \tilde{E}.
  \end{equation*}
  and the induction of the system on $\tilde{E}_1$ is canonically
  conjugated to $\left(\weightedtreeswithpath, \Sr, \mureg \right)$, thus is
  ergodic and (see \cref{sub:basicet}) so is the whole system.
\end{proof}

\begin{proof}[Proof of \cref{lem:exprh}]
  This proof can be found in \cite[Subsection~3.1]{aidekon_2007}.
  We reproduce it with our notations for the reader's convenience.
  From \cref{prop:slab}, we know that under $\Pa^{*}$,
  we have $\rh_0 = 0$ and the increments 
  $\rh_1$, $\rh_2 - \rh_1$, \dots,
  $\rh_{k+1} - \rh_k$, \dots  are i.i.d.
  For $n \geq 1$,
  \begin{equation*}
    \Pa^{*} \left(n \in \rh \right) 
    = \E^{*} \left[\# \rh \cap \{ 0, 1, \dotsc, n \}\right] 
    -  \E^{*} \left[\# \rh \cap \{ 0, 1, \dotsc, n-1 \}\right].
  \end{equation*}
  So, by the renewal theorem (\cite[p~360]{feller_vol2}), 
  \begin{equation*}
    \Pa^{*} \left(n \in \rh \right) \xrightarrow[ n \to \infty ]{ }
    1 / \E^{*} \left[\rh_1\right].
  \end{equation*}
  On the other hand, 
  \begin{align*}
    \Pa^{*} \left(n \in \rh \right)
    &=
    \Et \left[\beta \left(T\right)\right]^{-1} 
    \sum_{\abs{x}=n}^{}
    \Ea \left[\indic{x\in T}\indic{\tau_x <\infty}
      \indic{\tau_{\prt\root}>\tau_x}
      \indic{\forall k \geq \tau_x, X_k \neq \prt x}
    \right]
    \\
    &= 
    \sum_{\abs{x}=n}^{}
    \Ea \left[\indic{x\in T}\indic{\tau_x <\infty}
      \indic{\tau_{\prt\root}>\tau_x}
          \right]
    = \Pa \left[ \tau_{\prt\root} > \tau^{(n)} \right],
  \end{align*}
  where $\tau^{(n)} \coloneqq \inf \setof{k \geq 0}{\abs{X_k} = n}$.
  By dominated convergence,
  \begin{equation*}
    \Pa \left[ \tau_{\prt\root} > \tau^{(n)} \right] 
    \xrightarrow[ n \to \infty ]{  } \Pa \left[\tau_{\prt\root} = \infty\right] 
    = \Et \left[\beta (T) \right]. \qedhere
  \end{equation*}
\end{proof}

In order to construct a $\Se$-invariant measure on $\weightedtreeswithpath$, all
we need now is the right semi-conjugacy.
Let $h_{\mathsf{e}} : \tilde{E} \to \weightedtreeswithpath$ defined by
\begin{equation*}
  h_{\mathsf{e}} \left(t, \xa, i \right) \coloneqq
  \left(t \left[\ep_{i-1}\right] , \xa \left[\et_{i-1}\right] \right).
\end{equation*}
By construction,
\begin{equation*}
  h_{\mathsf{e}}\circ \St = \Se \circ h_{\mathsf{e}},
\end{equation*}
that is $h_{\mathsf{e}}$ is a semi-conjugacy on its image, so we get the desired ergodic system.
\begin{cor}
  The probability measure $ \muex \coloneqq 
  \tilde{\mureg}\circ h_{\mathsf{e}}^{-1}$ on 
  $\weightedtreeswithpath$ is invariant and ergodic with respect to the
  shift $\Se$.
\end{cor}
We now investigate further the law $\muex$.
Let $f : \weightedtreeswithpath \to \R_+$ a measurable function. By definition,
\begin{align*}
  \int f \left(t, \xa\right) \dd \muex \left(t, \xa\right)
  &= \int f \left(t \left[\ep_{i-1}\right] , \xa \left[\et_{i-1} \right] \right)
  \dd \tilde{\mureg} \left(t, \xa, i\right)
  \\
  &=\Et \left[\beta(T)\right]
  \sum_{i \geq 1}^{} \int \indic{\rh_1 \geq i} 
  f \left(t [\ep_{i-1}], \xa [\et_{i-1}]\right)
  \dd \mureg \left(t, \xa \right)
  \\
  &= \Et \left[\beta(T)\right] 
  \sum_{j \geq 1}^{}
  \sum_{i = 0}^{j-1}
  \E \left[\indic{\rh_1 = j, \, \tau_{\prt\root} = \infty}
    f \left(T \left[\ep_i\right],
      \Xa \left[\et_i \right] \right)
  \right].
\end{align*}
Reasoning on the value of $\rp_1$ and its strict ancestors, we get for all
$j \geq 1$,
\begin{align*}
  &\sum_{i = 0}^{j-1}
  \E \left[\indic{\rh_1 = j, \, \tau_{\prt\root} = \infty}
    f \left(T \left[\ep_i\right],
      \Xa \left[\et_i \right] \right)
  \right]
  \\
  &=  \sum_{x \in \words , \, \abs{x}=j}^{} \sum_{\root \preceq y \prec x}
  \sum_{s \geq 1}^{}
  \E \left[
    \indic{x \in T, \, \rp_1 = x, \, 
      \tau_{\prt\root} = \infty, \,
      \ep_{\abs{y}} = y, \, \et_{\abs{y}} = s}
    f \left(T [y], \Xa [s] \right)\right].
\end{align*}
Summing over $j$, we obtain 
\begin{align*}
  &\sum_{j \geq 1}^{}
  \sum_{i = 0}^{j-1}
  \E \left[\indic{\rh_1 = j, \, \tau_{\prt\root} = \infty}
    f \left(T \left[\ep_i\right],
      \Xa \left[\et_i \right] \right)
  \right]
  \\
  &=  \sum_{y \in \words}^{}
  \sum_{s \geq 1}^{}
  \E \left[
    \indic{y \in T, \, \tau_{\prt\root} = \infty,
      \ep_{\abs{y}} = y, \, \et_{\abs{y}} = s
    }
    f \left(T[y], \Xa[s]\right) 
    \sum_{x \succ y}^{} \indic{x \in T, \, \rp_1 = x}
  \right] 
  \\
  % line 6
  &=
  \sum_{y \in \words}^{}
  \sum_{s \geq 1}^{}
  \E \left[
    \indic{y \in T, \, \tau_{\prt\root} = \infty, \, \et_{\abs{y}} = s}
    f \left(T[y], \Xa[s]\right) 
    \indic{\rp_1 \succ y}
  \right]. 
\end{align*}
We want to use the Markov property at time $s$. For $s \geq 1$, and $y$ in
$\words$,
let $D_s(y)$ the event that:
\begin{itemize}
  \item the walk has not hit $\prt\root$ before time $s$ ;
  \item the walk hits $y$ at time $s$ and $\prt y$ at time $s-1$ ;
  \item for all $\root \prec z \preceq y$, there exist $1 \leq i < j \leq s$
    such that $X_i = z$ and $X_j = \prt z$.
\end{itemize}
For fixed $y \in \words$ and $s \geq 1$, on the event $\{y \in T\}$,
we denote
$\Xa'$ a random walk in $T$ starting from $y$ independant of $X_0, X_1, \dotsc,
X_s$ and 
$y^{-1} \Xa' \coloneqq \left(y^{-1} X'_0, y^{-1} X'_1, \dotsc \right)$. We
obtain
\begin{align*}
  &\indic{y \in T} 
    \Eq_\root^T \left[
      f \left(T[y], \Xa[s]\right)\indic{\forall k \geq s
      , \, X_k \neq \prt y}\indi{D_s(y)}
    \right]
    \\
  &=
  \indic{y \in T} 
    \Eq_y^T \left[f \left(T[y], y^{-1} \Xa'\right)\indic{\forall k \geq 0
    , \, X'_k \neq \prt y}\right]
  \Pq_\root^T\left[D_s(y) \right] . 
\end{align*}
We remark that the law of $y^{-1}\Xa'$ on the event
$\{ \forall k \geq 0, \, X'_k \neq \prt y \}$ is the same as the law of a random
walk $\Ya$ in the weighted tree $T[y]$, starting from $\root$, on the event
$\{ \forall k \geq 0, \, Y_k \neq \prt\root \}$, thus
\begin{align*}
  &\indic{y \in T} 
    \Eq_y^T \left[f \left(T[y], y^{-1} \Xa'\right)\indic{\forall k \geq 0
    , \, X'_k \neq \prt y}\right]
  \Pq_\root^T\left[D_s(y) \right] 
  \\
  &=\indic{y \in T} 
  \Eq_\root^{T[y]} \left[f \left(T[y], \Ya\right)\indic{\forall k \geq 0
    , \, Y_k \neq \prt y}\right]
  \Pq_\root^T\left[D_s(y) \right] 
  \\
  &=
  \indic{y \in T} 
    \Eq_\root^{T[y]} \brcd{f \left(T[y], \Ya \right)}{
    \tau_{\prt\root} (\Ya) = \infty}
    \Pq_\root^T\left[D_s(y), \forall k \geq s, X_k \neq \prt y
  \right] 
  \\
  &=\indic{y \in T}
    \Eq_\root^{T[y]} \brcd{f \left(T[y], \Ya \right)}
    {\tau_{\prt\root} (\Ya) = \infty}
  \Pq_\root^T\left[ \rp_1 \succ y, \, 
    \rt_1 = s, \,
  \tau_{\prt\root} = \infty \right]. 
\end{align*}
Summing over $s$, we obtain
\begin{align*}
  &\Et \left[\beta (T) \right]^{-1}
  \int f \left(t, \xa\right) \dd \muex \left(t, \xa\right)
  \\
  &=
   \sum_{y \in \words}^{}   
   \E \left[
     \indic{y \in T} 
     \Eq_\root^{T[y]} \brcd{f \left(T[y], \Ya \right)}
      {\tau_{\prt\root} (\Ya) = \infty}
    \Pq_\root^T\left[ \rp_1 \succ y, \, 
      \tau_{\prt\root} = \infty \right]
  \right].
\end{align*}

We may write
\begin{equation*}
  \Pq_\root^T\left[ \rp_1 \succ y, \tau_{\prt\root} = \infty \right]
  =
  \Pq_\root^{T^{\leq y} \gluetrees T[y]}\left[ \rp_1 \succ y, \tau_{\prt\root} = \infty \right]
  \eqqcolon h \left(T^{\leq y}, T[y] \right).
\end{equation*}
By the branching property, for any $y$ in $\words$,
\begin{align*}
  &\E \left[
     \indic{y \in T} 
     \Eq_\root^{T[y]} \brcd{f \left(T[y], \Ya \right)}
      {\tau_{\prt\root} (\Ya) = \infty}
    \Pq_\root^T\left[ \rp_1 \succ y, \, 
      \tau_{\prt\root} = \infty \right]
  \right]
  \\
  &=
  \E \left[\indic{y \in T} 
    \Eq_\root^{\tilde{T}} \brcd{f \left(\tilde{T}, \Ya \right)}
    {\tau_{\prt\root} (\Ya) = \infty}
  h \left(T^{\leq y}, \tilde{T}\right) \right],
\end{align*}
where $\tilde{T}$ is a weighted tree whose law is $\GW$, independant
of $T^{\leq y}$ and $\indic{y \in T}$.
As a consequence, the
conditional expectation of 
$ \indic{y \in T} h \left(T^{\leq y}, \tilde{T}\right)$
given $\tilde{T} = t$ equals
\begin{equation*}
  \Et \left[\indic{y\in T} h \left(T^{\leq y}, t\right)\right]
  \eqqcolon \kappa_y (t) . 
\end{equation*}
Summing over $y \in \words$, we finally obtain the following theorem which 
summarizes the results of this section.
\begin{thm}
  \label{thm:twpergo}
  The system
  $\left(\weightedtreeswithpath, \Se, \muex \right)$ is measure-preserving
  and ergodic.
  The probability measure $\muex$ has the following expression : for all
  non-negative measurable functions $f$,
  \begin{equation*}
    \int f \left(t, \xa \right) \dd \muex \left(t, \xa\right)
    =\frac{1}
          {\Et \left[\kappa(T)\right]}
    \E \left[
      f \left(T, \Xa\right) \kappa \left(T\right) \beta(T)^{-1} 
      \indic{\tau_{\prt\root} = \infty}
    \right]
  \end{equation*}
  where, for all weighted trees $t$,
  \begin{equation}
  \label{eq:defbetakappa}
    \kappa (t)
    \coloneqq
    \Et \bigg[
      \sum_{y \in T}^{}
        \Pq_\root^{T^{\leq y} \gluetrees t} \left( \rp_1 \succ y, \,
        \tau_{\prt\root} = \infty \right)
    \bigg]
    \quad \text{and} \quad
    \beta(t) \coloneqq \Pq^t_\root \left(\tau_{\prt\root} = \infty \right).
  \end{equation}
\end{thm}

\section{Invariant Measure for the Harmonic Flow Rule}
\label{sec:invariant_harm}

We now slightly change our point of view. We will forget everything about the
random path $\Xa$, except the ray it defines.
Let $\mu_\harm$ the projection on $\weightedtrees$ of the probability measure
$\muex$, that is, the probability defined by:
\begin{equation}\label{eq:muharm}
  \int f(t) \dd \mu_\harm (t)
  =
  \frac{1}
       {\Et\left[\kappa(T)\right]}
  \Et \left[f(T) \kappa(T) \right],
\end{equation}
for all non-negative measurable functions $f$ on $\weightedtrees$.
We denote $\weightedtreeswithrays$ the space of all weighted trees with a
distinguished ray, that is :
\begin{equation*}
  \weightedtreeswithrays \coloneqq
  \setof{ \left(t, \xi\right) }{ t \in \weightedtrees, \, \xi \in \rays t}.
\end{equation*}
We view it as a metric subspace of $\weightedtreeswithpath$.

We build a
Borel probability measure $\mu_\harm \ltimes \harm$ on $\weightedtreeswithrays$ by :
\begin{equation*}
  \int_{\weightedtreeswithrays}
  f \left(t, \xi\right) \dd \left(\mu_\harm \ltimes \harm\right)
  \left(t, \xi\right)
  =
  \int_{\weightedtrees}
  \left(  \int_{\rays t} f(t, \xi) \dd \harm_t \left(\xi\right)\right)
  \dd \mu_\harm \left(t\right),
\end{equation*}
for all positive measurable functions $f : \weightedtreeswithrays \to \R_+$.
The \emph{shift} $\Shift$ on $\weightedtreeswithrays$ is
\begin{equation*}
  \Shift \left(t, \xi \right)
  = \left(t \left[\xi_1\right] , \xi_1^{-1} \xi \right).
\end{equation*}
To check that this new system is (canonically) semi-conjugated to the one of
\cref{thm:twpergo} we need the following lemma.
\begin{lem}
  Let $t$ in $\weightedtrees$. Under the probability $\Eq_\root^{t}$,
  $\ray \left(\Xa\right)$ is independent of the event
  $\{\tau_{\prt \root} = \infty \}$.
\end{lem}
\begin{proof}
  Let $x$ in $t$.
  By the Markov property, first at time $\tau_{\prt\root}$ and then at time $1$,
  we have
  \begin{equation*}
    \Pq_\root^t \left(x \in \ray \left(\Xa\right), \,
    \tau_{\prt\root} < \infty \right)
    = \Pq_\root^t \left(x \in \ray \left(\Xa\right)\right)
    \Pq_\root^t \left(\tau_{\prt\root} < \infty \right).
  \end{equation*}
  Since the cylinders
  $\setof{\xi \in \rays t}{x \prec \xi}$, for $x$ in $t$,
  generate the Borel $\sigma$-algebra of $\rays t$, we conclude by a
  monotone class argument.
\end{proof}

\begin{prop}
  The system $\left(\weightedtreeswithrays, \Shift, \mu_\harm \ltimes \harm \right)$
  is measure-preserving and ergodic. Furthermore, the probability measure
  $\mu_\harm$ and $\GW$ are mutually absolutely continuous.
\end{prop}
\begin{proof}
  Let $h_{p \shortrightarrow r} : \weightedtreeswithpath \to \weightedtreeswithrays$ defined by
  $h_{p \shortrightarrow r}\left( t, \xa \right) = \left(t, \ray \left(\xa\right)\right)$.
  The mapping $h_{p
  \shortrightarrow r}$ is surjective and satisfies $h_{p \shortrightarrow
r}\circ \Se = \Shift \circ h_{p \shortrightarrow r}$, so is a
  semi-conjugacy. By the previous lemma, the probability measure
  $\mu_\harm \ltimes \harm$ equals $\muex \circ h_{p \shortrightarrow
  r}^{-1}$.

  We already know that $\mu_\harm$ is absolutely continuous with respect to
  $\GW$. We only need to show that, for $\GW$-almost every tree $t$, the density
  $\kappa(t)$ is positive. This is the case, because
  \begin{align*}
    \kappa (t) &=
    \Et 
    \bigg[
      \sum_{y \in T}^{}
        \Pq_\root^{T^{\leq y} \gluetrees t} \left( \rp_1 \succ y, \,
        \tau_{\prt\root} = \infty \right)
    \bigg]
    \\
    &\geq \Et 
    \bigg[
        \Pq_\root^{T^{\leq \root} \gluetrees t} \left( \rp_1 \succ \root, \,
        \tau_{\prt\root} = \infty \right)
    \bigg]
    = \Pq_\root^t \left(\tau_{\prt\root} = \infty\right)
    = \beta (t),
  \end{align*}
  and $\GW$-almost every tree $t$ is transient, thus is such that $\beta(t)>0$.
\end{proof}
We could also have used \cite[Proposition~5.2]{LPP95} to
prove the ergodicity and the absolute continuity of $\GW$ with respect to
$\mu_\harm$ since our measure $\mu_\harm$ was already known to be
absolutely continuous with respect to $\GW$.
To conclude that there indeed is a dimension drop phenomenon, we proceed as in
\cite[Theorem~7.1]{LPP95} and compare our flow rule $\harm$ to an other flow
rule called $\unif$. Before that, let us collect two equations about the flow
rule $\harm$.

\begin{lem}
  Let $t$ in $\weightedtrees$, and $\Xi$ a random ray in $\rays t$ whose law is
  $\harm_t$. Then,
  \begin{gather}
    \harm_T(i) = \Pq_\root^t \left(i \prec \Xi \right)
    = \frac{\wt(i)\beta \left(t[i]\right)}
           { \sum_{j=1}^{\nu_t{\root}} \wt(j)\beta \left(t[j]\right)};
    \label{eq:harmi}
    \\
    \beta (t)
    = \frac{\sum_{j=1}^{\nu_t(\root)} \wt(j)\beta(t[j])}
           {1 + \sum_{j=1}^{\nu_t(\root)}\wt(j)\beta(t[j])}.
    \label{eq:recharm}
  \end{gather}
\end{lem}
\begin{proof}
  We first apply the Markov property at time $1$ to get
  \begin{multline*}
    \Pq_\root^t \left(i \prec \Xi \right)
    =
    \frac{1}{1 + \sum_{j=1}^{\nu_t(\root)}\wt(j)}
    \bigg[
      \wt(i) \left( \Pq_i^t \left(\tau_\root = \infty\right)
        + \Pq_i^t \left(\tau_\root < \infty ,\, i \prec \Xi \right)
      \right)
      \\
      + \sum_{j \neq i}^{} \wt(j) \Pq_j^{t} \left(\tau_\root < \infty
      ,\, i \prec \Xi
      \right)
      + 1 \times \Pq_\root^t \left(i \prec \Xi \right)
    \bigg]. 
  \end{multline*}
  Using again the Markov property at time $\tau_\root$ yields the first formula.
  The proof of the second formula is similar.
\end{proof}

The uniform flow rule $\unif$ is defined as in $\cite{LPP95}$.
Let $W (T)$ the almost sure limit 
in $\intoo{0}{\infty}$
of $Z_n(T) / c_n$ as $n$ goes to infinity ;
$Z_n(T)$ being the number
of vertices of height $n$ in $T$, and the sequence $(c_n)_{n\geq 1}$ being the
(deterministic) Seneta-Heyde norming sequence of the reproduction law $\pp$ (see
\cite[chapter~5, Section~1]{LyonsPeres_book}).
The flow rule $\unif$ is defined on the first generation by
\begin{equation}
\label{eq:defunif}
  \unif_T \left(i \right) 
  = 
  \frac{W \left(T \left[i\right] \right)}
       { \sum_{\ell = 1}^{\nu_T\left(\root\right)} 
         W \left(T \left[\ell\right] \right) }
  , \quad \forall \, 1 \leq i \leq \nu_T \left(\root\right).
\end{equation}
For other generations, we use the rule
\begin{equation*}
  \unif_T \left(xy \right) = \unif_T (x) \unif_{T[x]} \left(y\right).
\end{equation*}
From the fact that $\lim_{n \to \infty} c_{n+1}/c_n = m$, we deduce the
recursive equation
\begin{equation}\label{eq:recunif}
  W \left(T\right) = \frac{1}{m} \sum_{\ell = 1}^{\nu_T(\root)} W
  \left(T[\ell]\right).
\end{equation}
Notice that the equations \eqref{eq:defunif} and \eqref{eq:recunif} are the
counterparts for $\unif$ of \eqref{eq:harmi} and \eqref{eq:recharm}.
It is known (see \cite[Section~6]{LPP95}) that, under the additional assumption
that $\E \left[N \log N\right] < \infty$, we have that almost surely
\[
  \hausd \unif_T = \log m = \hausd \rays T,
\]
that is, $\unif$ is maximal in some sense. However, we will not need this fact
(except the last equality, which is true without additional assumptions, see
\cite[Proposition~6.4]{Lyons_rwperco}).
\begin{lem}
  \label{lem:harmnequnif}
  For $ \GW$-almost any weighted tree $t$,
  $\harm_t \neq \unif_t$, unless $p_m = 1$ for some integer $m \geq 2$
  and the weights are all deterministic
  and equal.
\end{lem}
\begin{proof}
  We prove it by contradiction. By \cite[proposition~5.1]{LPP95},
  we may assume that almost surely, $\harm_T = \unif_T$.
  Let $k \geq 2$ and $i$ and $j$ distinct integers in $\intff{1}{k}$. We reason on the
  event $\{ \nu_T(\root) = k \}$, assuming it has positive probability (we
  recall that, by assumption, $\Pt \left(\nu_T(\root) = 1 \right) < 1$).
  Since $\harm_T(i) = \unif_T(i)$ and $\harm_T(j) = \unif_T(j)$, we have
  \begin{equation*}
    \frac{\wt[T](i) \beta \left(T[i]\right)}
           {W \left(T[i]\right)}
           = \frac{ \sum_{\ell = 1}^{k} \wt[T](\ell) \beta \left(T[\ell]\right)}
           {\sum_{\ell = 1}^{k} W \left(T[\ell]\right)}
           = \frac{\wt[T](j)\beta \left(T[j]\right)}
           {W \left(T[j]\right)}.
  \end{equation*}
  In particular,
  \begin{equation}\label{eq:betaw}
    \wt[T](i) \beta \left(T \left[i\right] \right) W \left(T [j]\right)
    = 
    \wt[T](j) \beta \left(T \left[j\right] \right) W \left(T [i]\right).
\end{equation}
  We first take the conditional expectation with respect to the $\sigma$-algebra generated
  by $\wt[T](i)$, $\wt[T](j)$ and the tree $T[i]$ to get that
  $\Et \left[W (T) \right] < \infty$, so it is $1$.
  Then, conditioning only with respect to $\wt[T](i)$ and
  $\wt[T](j)$, we get $\wt[T](i) = \wt[T](j)$. Let us denote
  $A_k \coloneqq \wt[T](1) = \dotsb = \wt[T](k)$. Simplifying in \eqref{eq:betaw} and
  taking the conditional expectation with respect to the subtree $T[i]$ gives
  $\beta \left(T[i] \right) = \alpha W \left(T[i] \right)$, for
  $\alpha = \Et \left[\beta \left(T\right)\right] \in \intoo{0}{1}$.
  Since the law of $T[i]$ is itself $\GW$, we have, for $\GW$-amost every
  tree $t$,
  \begin{equation*}
    \beta (t) = \alpha W\left(t\right).
  \end{equation*}
  We reason again on the event $\left\lbrace\nu_T(\root) = k\right\rbrace$.
  Using the recursive equations \eqref{eq:recharm} and \eqref{eq:recunif},
  we get
  \begin{equation*}
    \alpha W(T) = \beta (T) =
    \frac{A_k \alpha \sum_{j = 1}^{k} W \left(T[j]\right)}
         {1 + A_k \alpha \sum_{j = 1}^{k} W \left(T[j]\right)}
    = \frac{A_k \alpha m W (T)}
           {1 + A_k \alpha m W(T)}.
  \end{equation*}
  We then obtain
  \begin{equation*}
  m A_k \left(\alpha W\left(T\right) - 1 \right) = 1,
  \end{equation*}
  which, by independence, can only happen if $W$ and $A_k$ are almost surely
  constant. This is possible only if the law $\pp$ is degenerated and
  $k = m$. In this case, we have 
  $A_k = \frac{1}{k \left(\alpha - 1 \right)} > \frac{1}{k}$, that is,
  our random walk model reduces to transient $\lambda$-biased random walk on a
  regular tree, with deterministic $\lambda < m$.
\end{proof}

We now have all the ingredients to prove \cref{thm:ddrop}.

\begin{proof}[Proof of \cref{thm:ddrop}]
  This is very similar to \cite[Theorem~7.1]{LPP95}. We detail the proof for
  completeness and because of the fact that we work on weighted trees
  (although, as we will see, this does not make any difference here).
  For $(t,\xi)$ in $\weightedtreeswithrays$, let
  \begin{equation*}
  f(t,\xi) \coloneqq -\log \left(\harm_t (\xi_1)\right).
  \end{equation*}
  For any $k \geq 1$, we have
  \[
    f\circ \Shift^k (t,\xi) = 
    -\log\left(\harm_{t[\xi_k]}(\xi_k^{-1}\xi_{k+1})\right).
  \]
  On the other hand, using the flow-rule property of $\harm$, for any $n\geq 1$,
  \begin{equation*}
    \harm_t(\xi_n)
    =
    \prod_{k=0}^{n-1}
    \harm_{t[\xi_k]}(\xi_k^{-1}\xi_{k+1}).
  \end{equation*}
  Thus, by the ergodic theorem, and the fact that $\GW$ is absolutely continuous
  with respect to $\mu_\harm$, we have, for $\GW$-almost every weighted tree $t$
  , for $\harm_t$-almost every $\xi$,
  \begin{align*}
    -\frac{1}{n}\log \left(\harm_t(\xi_n)\right)
    \xrightarrow[ n\to\infty ]{  }
    &\int \left(\int f(t,\xi) \dd \harm_t(\xi)\right)\dd\mu_\harm(t)
    \\
    =&\frac{1}{\Et[\kappa(T)]}
    \Ea\left[-\log \left(\harm_T(\Xi_1)\right)\kappa(T)\right].
  \end{align*}
  Now, assume that our model does not reduce to $\lambda$-biased random walk
  on an $m$-regular tree.
  By Shannon's inequality (strict concavity of $\log$ and Jensen's
  inequality), for $\GW$-every $t$,
  \begin{equation}
    \label{eq:shannon}
    \sum_{i=1}^{\nu_t(\root)}
    \harm_t(i) (-\log) \left(\harm_t(i)\right)
    \leq
    \sum_{i=1}^{\nu_t(\root)}
    \harm_t(i) (-\log) \left(\unif_t(i)\right),
  \end{equation}
  with equality if and only if, for all $1 \leq i \leq \nu_t(\root)$,
  $\harm_t(i) = \unif_t(i)$. By \cref{lem:harmnequnif}, this happens with
  $\GW$-probability strictly less than $1$. Thus, integrating 
  \eqref{eq:shannon}
  with respect
  to $\mu_\harm$, which is equivalent to $\GW$, we get
  \[
    \frac{1}{\Et[\kappa(T)]}
    \Ea\left[-\log \left(\harm_T(\Xi_1)\right)\kappa(T)\right]
    <
    \frac{1}{\Et[\kappa(T)]}
    \Ea\left[-\log \left(\unif_T(\Xi_1)\right)\kappa(T)\right].
  \]
  This last term is the integral with respect to $\mu_\harm\ltimes\harm$
  of
  \begin{equation*}
    -\log \left(\frac{W(t[\xi_1])}{
    \sum_{i=1}^{\nu_t(\root)}W(t[i])}\right)
    = -\log \left(\frac{W (t[\xi_1])}{mW(t)}\right)
    = \log m + \log (W(t)) - \log( W(t[\xi_1])).
  \end{equation*}
For $(t,\xi)$ in $\weightedtreeswithrays$, let
$
  g(t,\xi) = W(t)
$.
We want to prove that $g - g\circ \Shift$ is integrable with integral $0$.
By the fact that our system is measure preserving and the ergodic theory Lemma~{6.2} in \cite{LPP95} (see also 
\cite[Lemma~17.20]{LyonsPeres_book}) all we need to show is that
$g - g\circ \Shift$ is bounded from below by an integrable function.
This is indeed the case because
\[
  \log W(t) - \log W(t[\xi_1])
  =
  \log \frac{ \sum_{i = 1}^{\nu_t(\root)} W(t[i]}{mW(t[\xi_1])}
  \geq -\log m,
\]
which concludes our final proof.
\end{proof}
\bibliographystyle{plain}
\bibliography{ddrop_rwre.bib}
\end{document}